\documentclass[oneside]{amsart}

\usepackage{geometry}                
\usepackage{graphicx}
\usepackage{amssymb}
\usepackage{amsmath}
\usepackage{amsthm}                
\usepackage{mathrsfs}  
\usepackage{bbm}
\usepackage{epstopdf}
\usepackage[makeroom]{cancel}
\usepackage{stmaryrd}
\usepackage{xargs}
\usepackage{mathdots}
\usepackage{mathtools}
\usepackage{tikz-cd}
\usepackage{relsize}
\usepackage{scalerel}
\usetikzlibrary{matrix}
\usetikzlibrary{arrows.meta}
\usepackage{changepage}
\usepackage{hyperref}
\usepackage[capitalise]{cleveref}
\usepackage{tikz}
\usetikzlibrary{calc}
\usepackage[shortlabels]{enumitem}

\newcommand{\bctikz}{$$\begin{tikzcd}}
\newcommand{\ectikz}{\end{tikzcd}$$}
\newcommand{\ntikz}{\end{tikzcd}\qquad\qquad\begin{tikzcd}}

  \usetikzlibrary{ 
  	cd,
  	math,
  	decorations.markings,
		decorations.pathreplacing,
  	positioning,
  	arrows.meta,
  	shapes,
		shadows,
		shadings,
  	calc,
  	fit,
  	quotes,
  	intersections,
    circuits,
    circuits.ee.IEC
  }

\tikzset{
	slash/.style={postaction={
  	decorate,
    decoration={markings, mark=at position 0.5 with
    	{\node[font=\footnotesize] {\tiny{\rotatebox{-10}{$/$}}};}}}
  }
}

\newtheorem{thm}{Theorem}[section]
\newtheorem*{thm*}{Theorem}
\newtheorem{lem}[thm]{Lemma}
\newtheorem{prop}[thm]{Proposition}

\theoremstyle{definition}
\newtheorem{defn}[thm]{Definition}
\newtheorem*{defn*}{Definition}
\theoremstyle{definition}
\newtheorem{ex}[thm]{Example}
\theoremstyle{remark}

\newtheorem{rem}[thm]{Remark}

\newcommand{\mdot}[1][1.7]{%
  \mathpalette{\CdotAux{#1}}\cdot%
}
\newdimen\CdotAxis
\newcommand*{\CdotAux}[3]{%
  {%
    \settoheight\CdotAxis{$#2\vcenter{}$}%
    \sbox0{%
      \raisebox\CdotAxis{%
        \scalebox{#1}{%
          \raisebox{-\CdotAxis}{%
            $\mathsurround=0pt #2#3$%
          }%
        }%
      }%
    }%
    \dp0=0pt %
    \sbox2{$#2\bullet$}%
    \ifdim\ht2<\ht0 %
      \ht0=\ht2 %
    \fi
    \sbox2{$\mathsurround=0pt #2#3$}%
    \hbox to \wd2{\hss\usebox{0}\hss}%
  }%
}
\newlength{\myline}
\newcommandx*{\triplearrow}[4][1=0, 2=1]{
  \draw[line width=\myline,double distance=3\myline,#3] #4;
  \draw[line width=\myline,shorten <=#1\myline,shorten >=#2\myline,#3] #4;
}

\newcommand{\cat}[1]{{\mathbf{#1}}} 
\newcommand{\smallcat}[1]{{\mathcal{#1}}} 
\newcommand{\cell}[1]{{\widehat{#1}}} 

\newcommand{\afm}{\leftarrow}
\newcommand{\emb}{\hookrightarrow}
\newcommand{\atol}{\xrightarrow}
\newcommand{\afml}{\xleftarrow}

\DeclareMathOperator{\Hom}{Hom}
\newcommand{\Homm}{\mathbf{Hom}}

\DeclareMathOperator{\id}{id}
\DeclareMathOperator{\Ob}{Ob}

\newcommand{\colim}{\mathop{\mathchoice{\textrm{colim}}{\mathlarger{\textrm{colim}}}{\scriptstyle \textrm{colim}}{\scriptscriptstyle \textrm{colim}}}\limits}
\renewcommand{\lim}{\mathop{\mathchoice{\textrm{lim}}{\mathlarger{\textrm{lim}}}{\scriptstyle \textrm{lim}}{\scriptscriptstyle \textrm{lim}}}\limits}
\newcommand{\coprodl}{\coprod\limits}
\newcommand{\sint}{\smallint}
\newcommand{\soint}{{\textstyle\oint}}

\newcommand{\nats}{\mathbb{N}}
\newcommand{\Set}{\cat{Set}}

\newcommand{\Cat}{\cat{Cat}}

\newcommand{\Prof}{\cat{Prof}}
\newcommand{\G}{\cat{G}}

\newcommand{\C}{\smallcat{C}}
\newcommand{\D}{\smallcat{D}}
\newcommand{\E}{\smallcat{E}}
\newcommand{\M}{\smallcat{M}}

\newcommand{\ch}{\cell{\C}}
\renewcommand{\dh}{\cell{\D}}

\renewcommand{\prod}{\mathop{\mathchoice{\mathlarger{\mathlarger{\Pi}}}{\mathlarger{\mathlarger{\Pi}}}{\scriptstyle \Pi}{\scriptscriptstyle \Pi}}}
\renewcommand{\sum}{\mathop{\mathchoice{\mathlarger{\mathlarger{\Sigma}}}{\mathlarger{\mathlarger{\Sigma}}}{\scriptstyle \Sigma}{\scriptscriptstyle \Sigma}}}

\newcommand{\starrows}{{\;\mathinner{\underset{t}{\overset{s}{\rightrightarrows}}}\;}}

\newlength\thetawidth
\newlength\thetaheight

\newcommand{\one}{\mathbbm{1}}
\newcommand{\two}{\mathbbm{2}}

\newcommand{\MC}{f\hspace{-.025cm}mc}
\newcommand{\fc}{f\hspace{-.03cm}c}
\newcommand{\fm}{f\hspace{-.025cm}m}
\newcommand{\fdc}{f\hspace{-.03cm}dc}
\newcommand{\fduc}{f\hspace{-.03cm}duc}
\renewcommand{\O}{{\mathbf{0}}}
\newcommand{\I}{{\mathbf{1}}}
\newcommand{\cplush}{{\cell{\C^+}}}

\title{Enrichment of Algebraic Higher Categories}
\author{Brandon Shapiro}
\date{}                                           

\begin{document}
\maketitle

\begin{abstract}
We provide a definition of enrichment that applies to a wide variety of categorical structures, generalizing Leinster's theory of enriched $T$-multicategories. As a sample of newly enrichable structures, we describe in detail the examples of enriched monoidal categories and enriched double categories, with a focus on monoidal double categories as broadly convenient bases of enrichment. 
\end{abstract}

\setcounter{tocdepth}{1}
\tableofcontents

\section*{Introduction}

In the beginning, \emph{$V$-enriched categories} were defined for $V$ a monoidal category (see for instance \cite{kelly}), showing that the compositional structure of a category could be expressed in settings where the collection $\Hom(a,b)$ of arrows from $a$ to $b$ is not (necessarily) a set but rather some other type of mathematical object. 

\begin{defn*}
For $(V,\otimes,I)$ a monoidal category, a $V$-enriched category consists of
\begin{itemize}
	\item a collection $A$ of objects
	\item for each $a,b \in A$ an object $\Homm(a,b)$ of $V$
	\item for each $a \in A$ a morphism $I \to \Homm(a,a)$ in $V$
	\item for each $a,b,c \in A$, a morphism $\Homm(a,b) \otimes \Homm(b,c) \to \Homm(a,c)$ in $V$
\end{itemize}
subject to certain unit and associativity equations.
\end{defn*}

This is useful not only for including in category theory the common feature of a category's Hom-sets admitting the extra structure of a group or space or category or what not, but also for allowing the defining features of categories to apply in settings that would otherwise seem entirely unrelated: famously, Lawvere showed that a slight generalization of metric spaces can be defined as simply categories in which $\Homm(a,b)$ is not a set but a non-negative real number representing the distance from $a$ to $b$.

Since then, the theory of enrichment has expanded in three directions, each seeking to answer one of the following questions:
\begin{enumerate}
	\item What structures, in lieu of categories, can \emph{be enriched}?
	\item What structures, in lieu of monoidal categories, can a structure \emph{be enriched in}? In other words, what are the possible \emph{bases of enrichment}?
	\item What are the different ways in which one fixed structure can be enriched in another fixed structure? 
\end{enumerate}

To the first question, many higher category structures such as $n$-categories and multicategories can be enriched in a symmetric monoidal category $V$, replacing the sets of top-dimensional cells with fixed boundary with objects of $V$.\footnote{We have been unable to find references for enriched algebraic $n$-categories, though the idea is far from new. Enriched multicategories in this sense are discussed in \cite[Section 2]{ElmendorfMandell} among others, and specialize Leinster's earlier definition of enriched multicategories.} Leinster's theory defines enriched $T$-multicategories for any cartesian monad $T$ (see \cref{sec:multicat} for the definition of $T$-multicategories). In this paper, we define enriched $T$-algebras for any familial monad $T$, which includes enrichment of most algebraic categorical structures in common use (see \cref{sec:familial} for the definition of a familial monad). 

We take the position that the fundamental idea of enrichment is varying the nature of the upper-dimensional parts of any algebraic structure. The algebraic structures we are interested in are higher categories of various flavors, and in this framework any applicable structure is regarded as a type of higher category. Higher categories are typically defined using sets and functions, which can be replaced by objects and morphisms of a different category or category-like structure. Leinster's theory applies to a wide variety of higher categories which can be modeled as generalized multicategories: this includes categories, plain multicategories, and virtual double categories, but not $n$-categories, monoidal categories, or double categories. These latter three structures are among those we highlight as examples of enrichable familial monad algebras.

To the second question, categories enriched in a multicategory are defined in \cite{LambekEnrichment}, similar to classical enrichment but with the identity and composition morphisms replaced by 0-to-1 and 2-to-1 morphisms in a multicategory. Categories enriched in a bicategory are described in \cite{BenBicat} (as ``polyads'') and \cite{variation}, among others. Leinster defines $T$-multicategories enriched in $T^+$-multicategories where $T^+$ is a new cartesian monad built from $T$ whose algebras are $T$-multicategories (see \cref{sec:enrichedmulticats} for a definition of $T^+$). In this paper, we define $T$-algebras enriched in any $T$-multicategory $V$. In the same way as Leinster's theory (which it generalizes only slightly), this type of enrichment base *feels* to be the most general possible structure with which $T$-algebras can be enriched with the associativity equations holding up to equality. It would be difficult to precisely state, let alone prove, this kind of statement, but we hope that like with Leinster's theory this feeling comes to be shared by many others. 

There are also ``weak'' analogues of enrichment, such as categories enriched in a monoidal bicategory where the usual equations hold up to higher cells (called ``enriched bicategories'' in \cite{enrichedbicat}).\footnote{Note that this is orthogonal to the previous notion of categories enriched in a bicategory. The former treats monoidal categories as 1-object bicategories while the latter treats them as locally discrete bicategories.} While it is not the main focus of this paper, as it tends to work in mostly the same way across the various other types of enrichment definitions, we very briefly discuss in \cref{weakenrichment} how one could weaken our definitions of enrichment so that the relevant equations only hold up to higher isomorphisms.

While $T$-multicategories are the most general type of enrichment base for $T$-algebras we consider, we also emphasize a range of specializations of $T$-multicategories. We describe (in \cref{sec:multicat}) $T$-multicategories which are \emph{trivial}, \emph{discrete}, and/or \emph{representable} relative to some restriction of the cell shapes in $T$-algebras, and discuss how enrichment in these settings can preserve additional structure and/or recover existing notions of enrichment in the literature. For example, for $\fc$ the free category monad on graphs: 
\begin{itemize}
	\item $\fc$-multicategories are virtual double categories
	\item representable $\fc$-multicategories are pseudo-double categories
	\item 0-discrete representable $\fc$-multicategories are bicategories
	\item 0-trivial $\fc$-multicategories are plain multicategories
	\item 0-trivial representable $\fc$-multicategories are monoidal categories
\end{itemize}

A recurring theme in our examples is the prevalence of monoidal double categories as a convenient base of enrichment. Double categories arise as the representable $\fc$-multicategories, and monoidal double categories are the representable $fmc$-multicategories for $fmc$ the free monoidal category monad on graphs (see \cref{fmc}), so monoidal double categories are natural bases of enrichment for categories and monoidal categories. Monoidal double categories also provide examples of Leinster's $fm$-multicategories for $fm$ the free multicategory monad (\cref{enrichedplainmulticats}). We also show in \cref{verticallytrivialrep} and \cref{vertenricheddoublecats} that monoidal double categories are a natural base for a certain form of enrichment of double categories. We list again the various structures enrichable in a monoidal double category in \cref{app:mdcat}.

To the third question, 
a major emphasis in this paper is placed on the relevance, when enriching structures more complex than categories, of how much low-dimensional information in an enrichable structure should exist or be preserved in the enrichment base. In our recurring example, we emphasize this choice when enriching monoidal categories in a monoidal double category $V$: the underlying monoid structure on the objects can be preserved strictly, weakly, or laxly in $V$, with only lax preservation guaranteed by the most general definition of enrichment.

Unlike many mathematical texts which introduce definitions in order to prove theorems, in this paper we prove results in order to introduce definitions. We believe that a general notion of enrichment for a variety of higher category structures is of independent interest, and so the provable claims we make are either to support the existence of the elements of that definition or show that in examples the definition recovers existing forms of enrichment. We leave the development of more theory for these enriched structures (along the lines of \cite{kelly}) to future work, though we have shown in \cite{mythesis} that in many circumstances the category of $T$-algebras enriched in the cartesian monoidal category of $T'$-algebras is equivalent to the Eilenberg-Moore category of yet another familial monad, which implies various convenient properties. We describe in \cite{adaptives} (joint with David Spivak) several examples of structures enriched in a monoidal double category, in the settings of machine learning and probability.

The novel technical advances presented here are admittedly modest. Most of the supporting definitions we include (such as $T$-multicategories, triviality, discreteness, representability, $T^+$) are covered or alluded to in \cite{leinster} or \cite{LeinsterEnrichment}, as are many of the supporting propositions. Our main contribution is to extend to all algebraic higher categories Leinster's definition of enrichment for generalized multicategories by incorporating the idea of higher and lower dimensional cell-shapes via a generalized construction of ``indiscrete'' higher dimensional structures from lower dimensional structures, and describing in detail how it works for various new examples (most notably monoidal categories and double categories). Our hope is that this self contained presentation of enrichment, leveraging recent approaches to analyzing familial monads and covering old and new examples, will provide a useful reference for a wide variety of enrichment practitioners.

\subsection*{Notation}

$\one$ denotes the terminal category, while $\two$ denotes the category $0 \to 1$.

For $\C$ a small category, $\ch$ denotes the category $\Set^{\C^{op}}$ of presheaves on $\C$. $*$ denotes the terminal presheaf in $\ch$.

All of the familial monad algebras and enriched structures we describe here will be small in the set theoretic sense, which we will not specify each time. Large enriched structures can be defined similarly without much additional work, but we do not discuss that any further.

\subsection*{Acknowledgements}

We would like to thank David Spivak for many enlightening conversations, and Marcelo Aguiar for sharing his previous work on duoidal enrichment of monoidal/2-/double categories (\cite{AguiarEnrichment}). 
This material is based upon work supported by the Air Force Office of Scientific Research under award number FA9550-20-1-0348.

\section{Familial Monads}\label{sec:familial}

\begin{defn}\label{def:familial}
A \emph{familially representable} (or simply \emph{familial}) monad $T$ on a presheaf category $\ch$ is a cartesian monad with functor part
$$TX_c = \coprod_{t \in T(1)_c} \Hom_\ch(T[t],X)$$
for each $c$ in $\C$, where $T(1) : \C^{op} \to \Set$ and $T[-] : \sint T(1) \to \ch$.
\end{defn}

A $T$-algebra is a presheaf $A$ on $\ch$ equipped with a map $h : TA \to A$ that, for each $t \in T(1)_c$, sends $a : T[t] \to A$ to $h(a) \in A_c$. Hence $h$ can be viewed as ``composing'' the arrangement of cells $a$ into a single $c$-cell in $A$. Hence we consider $t \in T(1)_c$ as an ``operation'' outputting a $c$-cell, and $T[t]$ as its ``arity.''

The monad structure on $T$ can be summarized by operations $\eta(c) \in T(1)_c$ for each $c$ in $\C$, and $\mu(t,f) \in T(1)_c$ for each $t \in T(1)_c$ and $f : T[t] \to T(1)$, with
$$T[\eta(c)] \cong y(c) \qquad\qquad \textrm{and} \qquad\qquad T[\mu(t,f)] \cong \colim_{x \in T[t]} T[f(x)],$$
satisfying naturality, unit, and associativity equations (see \cite[Theorem 2.2]{representability}).

\begin{ex}\label{fmc}
Consider the category $\G_1 = \O \starrows \I$. $\widehat{\G_1}$ is the category of graphs $X$, where $X_0$ is the set of vertices and $X_1$ the set of edges, each edge equipped with a source and target vertex. The standard example of a familial monad is the free category monad $\fc$ on $\widehat{\G_1}$, which is the identity on vertices and has arities given by the walking length $n$ path graph for all $n \ge 0$ (see \cite[Example 1.2]{representability} for more details).

Our main example of a familial monad, however, will be the free \emph{monoidal} category monad $\MC$ on graphs.\footnote{To be clear, this is the free strict monoidal category monad. There are also free weak and/or symmetric monoidal category monads on graphs, but we find this one the most illustrative for our definitions of enrichment.} Combining in a sense (which can be made precise in several different ways) the free monoid monad on sets and the free category monad on graphs, we define $\MC$ via its representing family $(\MC(1),\MC[-])$ as follows:
\begin{itemize}
	\item $\MC(1)_0 = \nats$, and $\MC(1)_1 = \coprod\limits_{n \in \nats} \nats^n$, with both the source and target maps $\MC(1)_1 \to \MC(1)_0$ sending $(n,m_1,...,m_n)$ to $n$
	\item $\MC[n]$ is the graph with $n$ vertices and no edges, and $\MC[n;m_1,...,m_n]$ is the disjoint union of walking paths of length $m_i$ for $1 \le i \le n$. $\MC[n]$ includes into this graph as the $n$ source vertices or as the $n$ target vertices, completing the definition of $\MC[-] : \sint \MC(1) \to \widehat{\G_1}$
	\item The monad structure on the functor
$$\MC (X)_0 = \coprod_{n \in \MC(1)_0} \Hom(\MC[n],X), \qquad \MC X_1 = \coprod_{(n;m_1,...,m_n) \in \MC(1)_1} \Hom(\MC[n;m_1,...,m_n],X)$$
is described by identifying the operations $1 \in \MC(1)_0$ and $(1;1) \in \MC(1)_1$ with identities on the 0-cells and 1-cells of $X$, and by observing that replacing each edge in $T[n;m_1,...,m_n]$ with a disjoint union of paths in a compatible way (as in, adjacent edges are plugged with the same number of disjoint paths) yields a new, usually larger disjoint union of usually longer paths.
\end{itemize}

$\MC$-algebras are precisely (small) strict monoidal categories, in which objects can be composed as in a monoid and arrows can be composed either from paths, as in a category, or from potentially disjoint pairs as in a monoidal category. The unique operations for each disjoint union of paths of any length witness that any such arrangement has a unique composite, which encodes the strict unit, associativity, and interchange equations of a monoidal category.
\end{ex}

\begin{ex}
Let $\fdc$ denote the free double category monad on double graphs, which are presheaves over $\G_1 \times \G_1$. In a double graph $X$, $X_{0,0}$ is the set of vertices, $X_{1,0}$ is the set of horizontal edges, $X_{0,1}$ is the set of vertical edges, and $X_{1,1}$ is the set of squares between parallel pairs of horizontal and vertical edges. Each double graph has a ``horizontal underlying graph'' and a ``vertical underlying graph'' given by the vertices and horizontal or vertical edges, respectively.

$\fdc$ acts as the free category monad on both $(X_{0,0},X_{1,0})$ and $(X_{0,0},X_{0,1})$, and inserts a square for every $n \times m$ grid of squares in $X$. In other words, $\fdc(1)_{i,j} = \nats^{i+j}$, $\fdc[() \in \fdc(1)_{0,0}]$ is a vertex, $\fdc[n \in \fdc(1)_{1,0}]$ is the string of $n$ composable horizontal edges, $\fdc[m \in \fdc(1)_{0,1}]$ is the string of $m$ composable vertical edges, and $\fdc[n,m]$ is the walking $n \times m$ grid of squares.
\end{ex}

\begin{ex}
Recall that a \emph{duoid} is a set with two monoidal structures (and no assumed compatibility between them).

A \emph{duoidal category} is a category $\M$ equipped with two monoidal structures $(\diamond,I)$ and $(\star,J)$ such that $J : \one \to \M$ and $\star : \M \times \M \to \M$ are lax monoidal functors with respect to $(\diamond,I)$ and the coherence transformations are $(\diamond,I)$-monoidal. (We will assume the monoidal structures are strict, but there is an analogous definition when they are weak.)

This lax monoidality means that, for instance, there is a not-necessarily-invertible morphism $I \to J$, a morphism $J \diamond J \to J$, a morphism $I \to I \star I$, and the interchange equation holds only up to a morphism
$$(a \star b) \diamond (c \star d) \to (a \diamond c) \star (b \diamond d).$$

Duoidal categories are algebras for a familial monad $\fduc$ on graphs, where $\fduc(1)_0$ is the free duoid on one object with arities sending each word to its set of variables, and $\fduc(1)_1$ is the set consisting of pairs of words in the free duoid on $\nats$ such that the second word is reachable by a finite sequence of the basic non-invertible duoidal structure maps of the four forms above. The arity of such a pair is the disjoint union of string graphs whose lengths are given by the natural numbers in the first word (which are the same as those in the second word, though perhaps in a different order).
\end{ex}

\subsection{Restriction of cell shapes}

Some documented forms of enrichment for higher dimensional structures include not just sets of lower dimensional cells, as in classical enriched categories, but lower dimensional algebraic structures. Enriched 2-categories, for instance, could be defined as a 1-category with an object of some $V$ for each pair of adjacent arrows, equipped with horizontal and vertical identity and composite maps. To facilitate this sort of definition, we describe how to extract these lower dimensional algebraic structures from a familial monad $T$ and related constructions.

\begin{defn}
A subcategory $\D$ of $\C$ is a \emph{restriction of cell shapes} if it is full and downward closed, meaning that for any $i : c' \to c$ in $\C$, if $c$ is in $\D$ then so is $c'$.
\end{defn}

The goal of this definition is to codify the properties we will need for a meaningful distinction between ``lower dimensional'' and ``higher dimensional'' cell shapes. Restrictions of cell shapes are sometimes called \emph{cribles} and are equivalent to functors $\C \to \two$ by taking the fiber over $0$. We will write $u : \D \to \C$ for the inclusion of such a subcategory.

There is an adjunction between $\ch$ and $\dh$ where the right adjoint $u^\ast$ restricts a presheaf on $\C$ to its $d$-cells for $d$ in $\D$, and the left adjoint $u_!$ sends a presheaf on $\D$ to the presheaf on $\C$ with no $c$-cells for all $c \in \Ob(\C) \backslash \Ob(\D)$. We say a presheaf in the image of $u_!$ (meaning one with no $c$-cells for $c$ not in $\D$) ``arises from $\dh$.''

\begin{ex}
$u : \G_0 \to \G_1$ is a restriction of cell shapes, where $\G_0$ is the terminal category containing the object 0 in $\G_1$. $\widehat{\G_0}$ is then the category of sets. $u^\ast$ sends a graph to its set of vertices, and $u_!$ sends a set to the graph with that set of vertices and no edges.
\end{ex}

\begin{ex}\label{endpoints}
Generalizing the above example, an object of $\C$ is an \emph{endpoint object} if it has no non-identity outgoing morphisms. A collection of endpoint objects in $\C$ determines a restriction of cell shapes given by the full subcategory of $\C$ on all of its other objects. An endpoint object is a ``top dimensional cell shape'' in a higher category structure, such as the arrow in categories, the $n$-cell in $n$-categories, and the $n$-to-1 arrow in a multicategory for all $n$.
\end{ex}

\begin{defn}
A familial monad on $\ch$ is \emph{$\D$-graded} for $\D$ a restriction of cell shapes if for all $d$ in $\D$, $t \in T(1)_d$, $T[t]$ arises from $\dh$.
\end{defn}

\begin{prop}\label{restriction_monad}
Given a $\D$-graded familial monad $T$ on $\ch$, there is a familial monad $T_\D$ on $\dh$ with $T_\D(1) = u^\ast T(1)$ and $$T_\D[-] : \sint u^\ast T(1) \to \sint T(1) \atol{T[-]} \ch \atol{u^\ast} \dh).$$
\end{prop}

\begin{proof}
This follows from the straightforward observation that for any $X$ in $\ch$, $T_\D u^\ast X = u^\ast TX$. As $u^\ast$ is a right adjoint and $u^\ast u_! Y = Y$ for $Y$ in $\dh$, this shows that $T_\D Y = u^\ast T u_! Y$, so $T_\D$ is the transport of $T$ along the adjunction $u^\ast \vdash u_!$.
\end{proof}

\begin{ex}
The monad $\MC$ on $\widehat{\G_1}$ is $\G_0$-graded, as all of the arities of operations in $\MC(1)_0$ have no edges. The monad $\MC_{\G_0}$ is easily seen to be the free monoid monad on sets, which we write as simply $\MC_0$.
\end{ex}

$u^\ast$ also has a right adjoint $u_\ast : \dh \to \ch$, where $u_\ast X_c = \Hom_\dh(u^\ast y(c), X)$. When $d$ is in $\D$, $u_\ast X_d \cong X_d$, and otherwise $X_c$ has a single $c$-cell in every position where a $c$-cell could possibly be inserted into $u^\ast X$ (that is, for each map $u^\ast y(c) \to u^\ast X$).

\begin{prop}
When $T$ is $\D$-graded, $u^\ast : \ch \to \dh$ and $u_\ast : \dh \to \ch$ both lift to functors between $T$-algebras and $T_\D$-algebras.
\end{prop}

\begin{proof}
As noted in the proof above, for a $T$-algebra $X$ we have $T_\D u^\ast X = u^\ast TX$, so the $T_\D$-algebra structure on $u^\ast X$ is given by applying $u^\ast$ to the structure map $TX \to X$ and the algebra equations follow easily.

Given a $T_\D$-algebra $Y$, we need to define a structure map $Tu_\ast Y \to u_\ast Y$, which for each $c$ in $\C$ unwinds to
$$\coprod_{t \in T(1)_c} \Hom_\ch(T[t],u_\ast Y) \to (u_\ast Y)_c.$$ 
Applying the adjunction $u_\ast \vdash u^\ast$ equates this map with one of the form
$$\coprod_{t \in T(1)_c} \Hom_\dh(u^\ast T[t], Y) \to (u_\ast Y)_c.$$
When $c$ is in $\D$, each $u^\ast T[t] = T_\D[t]$ and $(u_\ast Y)_c = Y_c$, so this map can be chosen to be the $c$-component of the of the $T_\D$-structure map on $Y$. When $c$ is not in $\D$, then any map $u^\ast T[t] \to Y$ precomposes with maps of the form $T_\D[i^\ast t] = u^\ast T[i^\ast t] \to u^\ast T[t]$ for $i : d \to c$ in $\C$ and $d$ in $\D$. This composite map $T_\D[i^\ast t] \to Y$ can be composed by the algebra structure into a $d$-cell in $Y$, and together (ranging over the different possible such $i$) these cells in $Y$ form the boundary of a potential $c$-cell in $u_\ast Y$. But there is exactly one $c$-cell in $u_\ast Y$ with that boundary, so the composite of the map $u^\ast T[t] \to Y$ in $(u_\ast Y)_c$ is uniquely determined. The algebra equations for this structure map then derive from those for the $T_\D$-algebra $Y$ and the uniqueness of $c$-cells in $u_\ast Y$ relative to their boundaries.
\end{proof}

\begin{rem}
While $u^\ast$ as a functor between categories of algebras also has a left adjoint, it is not given by $u_!$ on the underlying presheaves unless there is no $t \in T(1)_c$ for $c$ not in $\D$ such that $T[t]$ arises from $\dh$. Then for each such $t$ and map $T[t] \to Y$, a $c$-cell must be added to $u_! Y$ in order to define an algebra structure.
\end{rem}

\begin{ex}
$u_\ast : \widehat{\G_0} \to \widehat{\G_1}$ sends a set to the indiscrete graph on that set, with a single edge in each direction between any pair of vertices. As functors on algebras for $\MC$ and $\MC_0$: 
\begin{itemize}
	\item $u^\ast$ sends a strict monoidal category to its underlying monoid of objects
	\item $u_\ast$ sends a monoid to the indiscrete category on its underlying set, with the unique monoidal structure induced by the monoid structure on the objects. In particular, for the unique morphisms $a \to b$ and $a' \to b'$, there is a unique morphism $a \otimes a' \to b \otimes b'$, so the choice of their product is determined.
\end{itemize}
\end{ex}

\begin{ex}
Similarly, $\fduc_0$ is the free duoid monad on sets, the restriction functor $u^\ast$ on algebras sends a duoidal category to its underlying duoid, and $u_\ast$ sends a duoid to the indiscrete category on its objects, which inherits a canonical duoidal structure (albeit one that arises from a symmetric monoidal groupoid).
\end{ex}

\begin{ex}
Let $\G_1 \vee \G_1$ denote the full subcategory of $\G_1 \times \G_1$ spanned by $(0,0),(1,0),(0,1)$. The free double category monad $\fdc$ on double graphs is $(\G_1 \vee \G_1)$-graded, where presheaves on $\G_1 \vee \G_1$ are graphs with a single set of vertices but two distinct types of edges. This is because the operations in $\fdc(1)$ that output vertices or edges of either type do not have any squares in their arities, which are all strings of edges. The restriction of $\fdc$ to $\G_1 \vee \G_1$ from \cref{restriction_monad} is a monad whose algebras are pairs of categories with the same objects, and the restriction functor on algebras merely forgets the squares of a double category retaining only the underlying pair of categories.
\end{ex}

\begin{ex}
$\G_1 \times \G_1$ also has a full subcategory $\G_1 \times 0$ spanned by $(0,0),(1,0)$, isomorphic to $\G_1$. For this subcategory the ``vertical'' restriction of $\fdc$ is the free category monad on graphs, and the restriction functor on algebras sends a double category to its vertical category. Given a category $A$, $u_\ast A$ is the double category with a unique horizontal arrow between each pair of objects and a unique square between each pair of (now vertical) arrows in $A$.
\end{ex}

\section{$T$-Multicategories}\label{sec:multicat}

We now define $T$-multicategories and discuss their relationship with $T$-algebras. More details and references can be found in \cite{LeinsterEnrichment} and \cite{leinster}. Fix a familial monad $T$ on $\ch$.

\begin{defn}\label{def:multicat}
A $T$-multicategory $V = (V_0,V_1)$ is a span in $\ch$ of the form
\bctikz
& \ar{dl}[swap]{dom} V_1 \ar{dr}{cod} \\
TV_0 & & V_0
\ectikz
equipped with unit and multiplication maps (fixing $X$ and $TX$) from the following two spans
\bctikz
& \ar{dl}[swap]{\eta} V_0 \ar[equals]{dr} \\
TV_0 & & V_0
\ntikz
& \ar{dl}[swap]{\mu \circ T(dom) \circ \pi_1} TV_1 \times_{TV_0} V_1 \ar{dr}{cod \circ \pi_2} \\
TV_0 & & V_0
\ectikz
satisfying unit and associativity equations. 
\end{defn}

Unwinding this definition, a $T$-multicategory consists of:
\begin{itemize}
	\item $V_0$ in $\ch$, whose elements in $(V_0)_c$ can be regarded as $c$-cells
	\item for each $c$ in $\C$ and $t \in T(1)_c$, cells (in $(V_1)_c$) resembling ``arrows'' from a map $T[t] \to V_0$ to a $c$-cell in $V_0$, which we call $t$-arrows
	\item for each $i : c' \to c$ in $\C$ and $t \in T(1)_c$, every $t$-arrow restricts along $i$ to a $T(1)_i t$-arrow. There is hence a canonical map $V_1 \to T(1)$ in $\ch$ sending all $t$-arrows to $t \in T(1)_c$
	\item for each $c$ in $\C$ and $a \in (V_0)_c$, an identity $\eta(c)$-arrow from $a$ to $a$
	\item for each $c$ in $\C$, $t \in T(1)_c$, $f : T[t] \to T(1)$, $t$-arrow $\alpha$, and $\beta: T[t] \to V_1$ commuting over $T(1)$, a $\mu(t,f)$-arrow from the combined domains of $\beta$ in $X$ to the codomain of $\alpha$
	\item these identities and composites satisfy unit and associativity equations
\end{itemize}

For convenience, for $c$ in $\C$ we write $V^c$ for $(V_0)_c$ and for $t \in T(1)_c$ we write $V^t$ for the subset of $(V_1)_c$ containing the $t$-arrows. $V$ can in fact be regarded as a presheaf over a category $\C^{+T}$ which is an algebra for a familial monad $T^+$ on $\widehat{\C^{+T}}$, which we discuss in \cref{sec:enrichedmulticats}.

\begin{prop}
For $u : \D \to \C$ a restriction of cell shapes, when $T$ is $\D$-graded a $T$-multicategory $V = (V_0,V_1)$ restricts to a $T_\D$-multicategory $u^\ast V = (u^\ast V_0, u^\ast V_1)$.
\end{prop}

\begin{proof}
This is a straightforward observation from the unwinded definition above, as if $T$ is $\D$-graded the identities and composites of $d$-cells in $V$ and $t$-arrows for $d$ in $\D$ and $t \in T(1)_d$ are unaffected by forgetting the higher dimensional cell shapes and operations.
\end{proof}

\begin{ex}
Recall (from, for instance, \cite[Section 2.1]{LeinsterEnrichment}) that a $\fc$-multicategory is what is sometimes called a virtual double category: a structure resembling a double category but without horizontal composition of horizontal morphisms or squares. Instead of horizontal composition, squares can have a string of adjacent horizontal morphisms in the top row as in \eqref{vdoub} on the right, and many such horizontally adjacent squares can be composed vertically above a single square: in \eqref{vdoub}, the arrangement on the left composes into a 3-to-1 square as on the right:
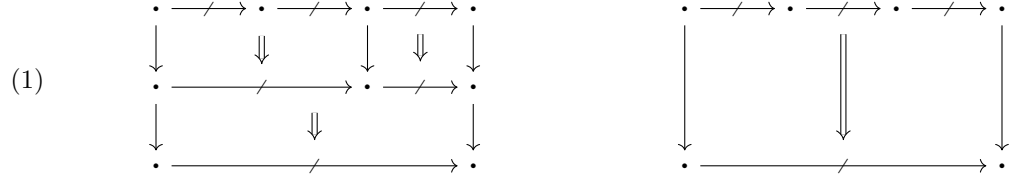
\begin{equation}\label{vdoub}\begin{tikzcd}[column sep={40,between origins}]
\mdot \dar \ar[slash]{r} & \mdot \ar[slash]{r} & \mdot \dar \ar[slash, ""{name=D, below}]{r} & \mdot \dar & & \mdot \ar[slash]{r} \ar{dd} & \mdot \ar[slash,""{name=F, below}]{r} & \mdot \ar[slash]{r} & \mdot \ar{dd} \\
\mdot \dar \ar[slash, ""{name=A, above}]{rr} & {\color{white}{\mdot}} \ar[phantom, ""{name=B, below}]{r} & \mdot \ar[slash, ""{name=E, above}]{r} & \mdot \dar \\
\mdot \ar[slash, ""{name=C, above}]{rrr} & & & \mdot & & \mdot \ar[slash,""{name=G, above}]{rrr} & & & \mdot
\arrow[Rightarrow,shorten=4,from=1-2,to=A]
\arrow[Rightarrow,shorten=5,from=B,to=C]
\arrow[Rightarrow,shorten=5,from=D,to=E]
\arrow[Rightarrow,shorten=5,from=F,to=G]
\end{tikzcd}\end{equation}

As an $\fc$-multicategory, the $\O$- and $\I$-cells are the objects and horizontal morphisms of the virtual double category, the $\eta(\O)$-arrows are the vertical morphisms, the $n$-arrows for $n \in \fc(1)_1$ are the $n$-to-1 squares, and the composites and identities agree with the composition and vertical identities in a virtual double category. Under both descriptions of this structure, the unit and associativity equations are the same.
\end{ex}

\begin{rem}
To make the choice of directions unambiguous when there are many different types of arrows present in a $T$-multicategory, we will always say that the $t$-arrows point in the ``forward'' direction. So in a virtual double category, the vertical morphisms and $n$-to-1 squares all point in the forward direction in this terminology.
\end{rem}

\begin{ex}
An $\MC$-multicategory consists of the following:
\begin{itemize}
	\item A graph $V_0$ with vertex set $V^\O$ and edge set $V^\I$
	\item For each $n \in \MC(1)_0$, a set $V^n$ of arrows from a list of $n$ vertices in $V^\O$ to a single vertex. These arrows resemble the many-to-1 arrows in a multicategory
	\item For each $(n;m_1,...,m_n) \in \MC(1)_1$, a set $V^{n;m_1,...,m_n}$ of arrows from an arrangement of paths of length $m_i$ in $V_0$ for $1 \le i \le n$ to a single edge. These arrows have source and target $n$-to-1 arrows in $V^n$. When $n=1$, these resemble the many-to-1 squares in a virtual double category, and in general we call them ``many-to-1 multi-squares''
	\item For each $x \in V^\O$, an identity arrow from $x$ to $x$ in $V^1$
	\item For each edge $x \in V^\I$, an identity arrow from $x$ to $x$ in $V^{1;1}$
	\item Compositions of many-to-1 arrows between vertices that form a multicategory
	\item Arrows from arrangements of paths to an edge have compositional structure similar to that of a virtual double category, which respects the multicategory structure on their sources and targets
\end{itemize}

In other words, it consists of a plain multicategory, a set of edges between the objects of that multicategory, and a virtual-double-category-like structure on those edges. When those edges have the structure of a category that extends to the ``squares,'' this looks like the double multicategories of \cite[Definition 3.10]{doublemulti}.

Given such an $\MC$-multicategory $V$, $u^\ast V$ is its underlying multicategory on the vertices $V^\O$, which is precisely a $\MC_0$-multicategory (\cite[Example 4.2.7]{leinster}).
\end{ex}

\begin{ex}
An $\fdc$-multicategory $V$ consists of:
\begin{itemize}
	\item A double graph $V_0$
	\item A virtual double category $V_h$ whose underlying graph is the horizontal underlying graph of $V_0$
	\item A virtual double category $V_v$ whose underlying graph is the vertical underlying graph of $V_0$
	\item ``(n,m)-arrows'' from an $n \times m$ grid of squares in $V_0$ to a single square in $V_0$, whose vertical source and target are squares in $V_h$ and whose horizontal source and target are squares in $V_v$
	\item Identities and composites of these 3-dimensional arrows analogous to those in a virtual double category, respecting the identities and composites in $V_h,V_v$ and satisfying the usual unit and associativity axioms
\end{itemize}
It would also make sense to call this a ``virtual triple category.''
\end{ex}

We now define $T$-multifunctors and transformations between them.

\begin{defn}
A \emph{$T$-multifunctor} is a map of spans determined by maps on $V_0$ and $V_1$, which commutes with identities and composites.
\end{defn}

Concretely, for $T$-multicategories $V,V'$, the map of spans amounts to maps $V^c \to V'^c$ and $V^t \to V^t$ natural in $c$ and $t$.

\begin{defn}
For $T$-multifunctors $F,G : V \to V'$, a \emph{$T$-multinatural transformation} (or simply \emph{transformation}) is a natural assignment of, for all $t$-arrows in $V$ from $a : T[t] \to V_0$ to $b \in V^c$, a $t$-arrow in $V'$ from $Fa : T[t] \to V_0 \to V'_0$ to $Gb \in V'^c$. Here natural means that these arrows are closed under precomposition with arrows under $F$ and postcomposition with arrows under $G$.
\end{defn}

\subsection{Trivial $T$-multicategories}

In \cref{sec:enrichment}, we define $T$-algebras enriched in a $T$-multicategory $V$. In order to recover the classical definitions of enrichment of categories and other fundamental categorical structures, we describe in the next few subsections various specializations of the notion of $T$-multicategory that recover more familiar bases of enrichment, such as monoidal categories. This process is a generalization of \cite[Table 2.1]{LeinsterEnrichment}, where we provide definitions of the conditions ``vertically discrete'' (which we call 0-discrete), ``vertically trivial'' (which we call 0-trivial), ``representable,'' and ``uniformly representable'' listed in that table for the case when $T$ is familial.  

When $\fc$-multicategories are specialized to plain multicategories (see \cite[Example 2.1.1.v]{LeinsterEnrichment}), it is assumed that the sets of objects (or in our terminology, vertices) and vertical arrows (here $\O$-arrows) have only one element. This reduces the dimension of the structure by allowing the set of horizontal arrows (here $I$-cells), each of which has the unique object as its source and target, to instead be treated like objects and the many-to-one squares (here $n$-arrows) like many-to-one arrows. This process of imposing that there is only one cell and one arrow of certain lower-dimensional shapes in a $T$-multicategory applies broadly to specializing $T$-multicategories to more classical or classically-inspired bases of enrichment.

\begin{defn}
For $u : \D \to \C$ a restriction of cell shapes and $T$ a $\D$-graded familial monad on $\ch$, a $T$-multicategory $V$ is \emph{$\D$-trivial} if for each $d$ in $\D$, there is a unique $d$-cell and for each $t \in T(1)_d$, there is a unique $t$-arrow. When $V$ is $\C$-trivial, we say that $V$ is simply \emph{trivial}.
\end{defn}

It is important to assume that $T$ is $\D$-graded, as otherwise for $t \in T(1)_d$ the arity $T[t]$ could contain $c$-cells for $c$ not in $\D$, so it would not be clear what the source diagram of the unique $t$-arrow in a $\D$-trivial $T$-multicategory should be. 

\begin{ex}
For the restriction $\G_0 \to \G_1$, a $\G_0$-trivial $\fc$-multicategory, which we call simply 0-trivial, has a unique 0-cell and $\eta(\O)$-arrow. The only remaining pieces of data are the 1-cells and $n$-arrows, which form the objects and $n$-to-1 arrows of a plain multicategory as discussed above.
\end{ex}

\begin{ex}
A 0-trivial $\MC$-multicategory is quite similar to a 0-trivial $\fc$-multicategory. In place of the $n$-to-1 squares in the latter that are interpreted as $n$-to-1 arrows in a plain multicategory, in a 0-trivial $\MC$-multicategory $V$ the set $V^{n;m_1,...,m_n}$ of $(m_1,...,m_n)$-to-1 multi-squares can be interpreted as $(m_1+\cdots+m_n)$-to-1 arrows. Unlike the multi-arrows in a plain multicategory however, the many 1-cells (interpreted as objects) are arranged arranged in $n$ different rows rather than 1, and such a multi-arrow can only be precomposed with multi-arrows into each domain object such that the arrows into objects of the same row each have the same number of rows in their domains.
\end{ex}

\begin{ex}
When $u : \D \to \C$ is a restriction of cell shapes away from a choice $e = \{e_k\}$ of endpoint objects (\cref{endpoints}) and $T$ is a $\D$-graded familial monad on $\ch$, $\D$-trivial $T$-multicategories are similar to plain multicategories: the objects are $e_k$-cells for all of the chosen endpoints $e_k$ in $\C$, and the many-to-1 arrows go from $e_k$-cells typed according to the $e$-cells in $T[t]$ to an $e_{k'}$-cell for $t \in T(1)_{e_{k'}}$. 

We say $T$ is \emph{finitary over $e$} if $T$ is $\D$-graded and for all $t \in T(1)_e$, $T[t]_e$ is finite.
By \cite[Corollary C.4.8]{leinster}, when $T$ is finitary over $e$ any symmetric multicategory forms a $\D$-trivial $T$-multicategory, albeit in a non-canonical way.
\end{ex}

\subsection{Discrete $T$-multicategories}

For a $T$-algebra $h : TA \to A$, there is a corresponding $T$-multicategory $MA$ given by the span
\bctikz
& \ar[equals]{dl} TA \ar{dr}{h} \\
TA & & A.
\ectikz
This $T$-multicategory has a single $t$-arrow from $a$ to $h(a)$ for each $c$ in $\C$, $t \in T(1)_c$, and $a : T[t] \to A$, with identities and composites uniquely determined by their domains.

This construction can be interpreted as replacing the algebraic structure of $A$ with more geometric ``witnesses'' in $MA$ in the form of the unique arrow from $a$ to $h(a)$ for each composable arrangement $a : T[t] \to A$. The algebraic structure in $MA$ is then comparatively simple, merely recognizing that witnesses to a nested composition in $A$ assemble into a witness of the corresponding total composition.

$T$-multifunctors $MA \to V$, for $A$ a $T$-algebra, tend to behave like lax functors out of $A$ when $A$ is interpreted as a category-like structure.

\begin{ex}
Given a strict monoidal category $A$, $MA$ is the following $\MC$-multicategory:
\begin{itemize}
	\item The vertices and edges are those of the underlying graph of $A$
	\item There is an $n$-to-1 arrow from $a_1,...,a_n$ to $b$ precisely when $b = a_1 \otimes \cdots \otimes a_n$
	\item There is an arrow from the paths 
$$a_{1,1} \to \cdots \to a_{1,m_1} \qquad \cdots \qquad a_{n,1} \to \cdots \to a_{n,m_n}$$
to the edge $b \to b'$ precisely when $b = a_{1,1} \otimes \cdots \otimes a_{n,1}$, $b = a_{1,m_1} \otimes \cdots \otimes a_{n,m_n}$, and $b \to b'$ is the product of the composites of the paths
	\item Identities are given by the unique arrow from $a$ to $a$, where $a$ is any vertex or edge
	\item Each arrangement of paths has a unique outgoing arrow, which determines arrow composition. This composition is possible because of the algebra structure on $A$ which ensures every arrangement has a composite
\end{itemize}

A $\MC$-multifunctor from $MA$ to $V$ then consists of a vertex $v_a$ of $V$ for each object of $A$, and $n$-to-1 cells from $v_{a_1},...,v_{a_n}$ to $v_b$ whenever $b = a_1 \otimes \cdots \otimes a_n$ (and similar structure for paths of morphisms). This is the sense in which this resembles a lax monoidal functor, as a statement like ``$v_b = v_{a_1} \otimes \cdots \otimes v_{a_n}$'' may not even make sense in $V$ but there can be an arrow from a list of objects to another object. Below we describe conditions on $V$ such that these statements actually do make sense.

However, $\MC$-multifunctors $MA \to MA'$ precisely correspond to strict monoidal functors $A \to A'$, as the action on arrows ensures that the assignment on objects and morphisms respects composition and products. 
\end{ex}

Just like the notion of triviality of a $T$-multicategory, $T$-algebras are characterized among $T$-multicategories by a uniqueness condition on the arrows of $V$.

\begin{defn}
For $u : \D \to \C$ a restriction of cell shapes, a $T$-multicategory $V$ is \emph{$\D$-discrete} if for each $d$ in $\D$, $t \in T(1)_d$, and $a : T[t] \to V_0$, there is a unique $t$-arrow with domain $a$. When $V$ is $\C$-discrete, we say that $V$ is simply \emph{discrete}.
\end{defn}

The name ``discrete'' corresponds to Leinster's notion of ``vertically discrete'' $\fc$-multicategories, emphasizing that the forward arrows to $\D$-cells in a discrete $T$-multicategory only encode $T_\D$-algebra structure (if there is any, unlike $\fc_0 = \id$), rather than additional data.

\begin{prop}
$V$ is $\D$-discrete if and only if $u^\ast V \cong MA$ for some $T_\D$-algebra $A$.
\end{prop}

\begin{proof}
If $V$ is $\D$-discrete, $u^\ast V_1$ is isomorphic to $T_\D A$, and it is easy to see that such a $T_\D$-multicategory is isomorphic to one of the form $MA$, where $A = u^\ast V_0$ and the codomain map 
$$T_\D u^\ast V_0 \cong u^\ast V_1 \to u^\ast V_0$$
provides the algebra structure. The converse is immediate from the definition.
\end{proof}

\begin{ex}
An $\MC$-multicategory $V$ is therefore discrete if it has the form $MA$ for $A$ a monoidal category, and $\G_0$-discrete  (or simply 0-discrete) if there is a monoid structure on its vertices and it has only unique $n$-to-1 arrows from $a_1,...,a_n$ to the product $a_1 \cdots a_n$.
\end{ex}

\subsection{Representable $T$-multicategories}

A $t$-arrow could be interpreted as a witness to the fact that its domain can compose into its codomain, as in the discrete $T$-multicategories, but in practice they often resemble something more like ``the composite of the domain maps into the codomain.'' We now make this latter interpretation precise by defining the corresponding condition on a $T$-multicategory. 

We now describe weaker conditions in which each $a$ can be the domain of many $t$-arrows, but with one or more still distinguished as witness(es) to composition.

\begin{defn}
In a $T$-multicategory $V$, for $c$ in $\C$, $t \in T(1)_c$, and $a : T[t] \to V_0$, a \emph{universal $t$-arrow for $a$} is a $t$-arrow $h_a$ with domain $a$ such that every $t$-arrow with domain $a$ factors uniquely as $h_a$ composed with an $\eta(c)$-arrow with domain $cod(h_a)$.
\end{defn}

This definition, as well as the following one, should be compared with universal arrows in representable multicategories as in \cite{Hermida}, which resemble the arrows from $a_1,...,a_n$ to $a_1 \cdots a_n$ in a discrete $T$-multicategory but in a setting where there may still be additional arrows out of the list $a_1,...,a_n$, as in the classical underlying multicategory of a monoidal category. 

\begin{defn}
For $u : \D \to \C$ a restriction of cell shapes, a $T$-multicategory $V$ is \emph{$\D$-representable} if:
\begin{itemize}
	\item for each $d$ in $\D$, $t \in T(1)_d$, and $a : T[t] \to V_0$, $a$ has a universal $t$-arrow
	\item universal arrows in $u^\ast V$ are closed under composition
\end{itemize}
$V$ is furthermore \emph{uniformly $\D$-representable} if each such $a$ is equipped with a choice of universal $t$-arrow $h_a$, closed under identities and composition. When $V$ is (uniformly) $\C$-representable, we say that $V$ is simply \emph{(uniformly) representable}.
\end{defn}

Note that identity arrows are automatically universal, and more generally universal arrows can be interpreted as a generalization of isomorphisms for $t$-arrows in $V$. Uniformly representable $T$-multicategories describe the situation of $t$-arrows corresponding to maps out of the composite of the domain, while representable $T$-multicategories provide a weaker version of this picture where an arrangement $a$ can have multiple valid composites, though all must be isomorphic.

\begin{prop}\label{representable_has_algebra}
A uniformly $\D$-representable $T$-multicategory $V$ induces a $T_\D$-algebra structure on $u^\ast V_0$.
\end{prop}

\begin{proof}
For $t \in T(1)_d$, $h : T_\D u^\ast V_0 \to u^\ast V_0$ composes $a : T[t] \to V_0$ into the codomain of $h_a$, and these compositions respect $\eta$ and $\mu$ as the universal arrows $h_a$ are closed under identities and composition.
\end{proof}

\begin{ex}\label{representable_fc}
As in \cite[Examples 2.1.1.i and 2.1.1.ii]{LeinsterEnrichment}, an $\fc$-multicategory $V$ is representable when there is a horizontal composition operation on 1-cells and squares making $V$ into a pseudo-double category, and uniformly representable when it is a strict double category. In the strict case, the choice of universal $n$-arrows for each string of $n$ adjacent 1-cells determines their composition, while horizontal composition of $n$ squares is induced by composing the codomain 1-cells in this manner and then factoring the resulting $n$-to-1 square through the universal $n$-arrow out of the domain 1-cells. When $V$ is merely representable, any choice of universal $n$-arrow for each $n$ adjacent 1-cells suffices to define the composition, which will be associative and unital up to (1-arrow) isomorphism as the universal $n$-arrows out of a fixed choice of $n$ adjacent 1-cells are unique up to isomorphism.
\end{ex}

\begin{ex}
A uniformly representable $\MC$-multicategory is given by a strict monoidal double category: all of the $n$-to-1 arrows from $a_1,...,a_n$ to $b$ are given by vertical morphisms $a_1 \otimes \cdots \otimes a_n \to b$, and all of the arrows from an arrangement of paths to an edge are given by squares from the product of the composites of those paths to an edge.

A representable $\MC$-multicategory corresponds to a weak monoidal pseudo-double category, where there may be many universal arrows out of a list of vertices $a_1,...,a_n$ but all are canonically isomorphic, with similar behavior for composites and products of edges.
\end{ex}

\begin{ex}\label{doublemulticats}
The double multicategories of \cite[Definition 3.10]{doublemulti} don't fit neatly into these properties of a $\MC$-multicategory, because they are those which are ``uniformly representable in the category direction but not representable in the monoidal direction.'' In other words, they admit horizontal composition of 1-cells and multi-squares, but the multicategories of 0-cells and $n$-arrows, and of 1-cells and $(n;1,...,1)$-arrows do not arise from monoidal categories.

However, double multicategories are precisely the uniformly representable $\fm$-multicategories for $\fm$ the free plain multicategory monad on graphs with many-to-1 arrows. As double multicategories are defined as categories internal to plain multicategories, this is a straightforward example of the following characterization of (uniformly) representable $T$-multicategories.
\end{ex}

\begin{prop}\label{representable_structured}
Uniformly representable $T$-multicategories are equivalent to categories internal to $T$-algebras.
\end{prop}

This claim is analogous to the same result equating representable plain multicategories and monoidal categories in \cite[Section 9]{Hermida}. The proof is entirely analogous, but we sketch the correspondence below. The relationship between $T$-multicategories and categories internal to $T$-algebras (also called ``$T$-structured categories'') is discussed further in \cite[Section 6.6]{leinster}.

\begin{proof}
Given a $T$-structured category $(C_0,C_1)$, we can define its underlying uniformly representable $T$-multicategory as follows:
\begin{itemize}
	\item its $c$-cells are the same as $C_0$
	\item for $t \in T(1)_c$, its $t$-arrows from $a : T[t] \to C_0$ to $b \in (C_0)_c$ are the $c$-cells in $C_1$ with source the composite of $a$ and target $b$
	\item identities and composites derive directly from the category structure on $(C_0,C_1)$; as a compatible arrangement of arrows corresponds to a $c$-cell in $C_1$ from the composite of $a : T[\mu(t,f)] \to C_0$ to the composite of $b : T[t] \to C_0$, which can be composed in the categorical (forward) direction with a $t$-arrow from the composite of $b$ to another $c$-cell
	\item the distinguished universal arrows are the $t$-arrows given by the identities in $C_1$ on the composite of $T[t] \to C_0$, which are closed under $T$-multicategory composition as identities are closed under composition in an internal category
\end{itemize}

Conversely, given a uniformly representable $T$-multicategory $V$, we define its corresponding $T$-structured category $(C_0,C_1)$ by:
\begin{itemize}
	\item $C_0$ is precisely $V_0$, consisting of the $c$-cells of $V$, which forms a $T$-algebra by \cref{representable_has_algebra}
	\item $C_1$ has $c$-cells given by the $\eta(c)$-arrows in $V$, with source and target $c$-cells in $C_0$ defined as the source and target of the arrow in $V$
	\item the $T$-algebra structure on $C_1$ is defined as in \cref{representable_fc}, where $a : T[t] \to C_1$ is composed by first composing its target $t(a) : T[t] \to C_1 \to C_0$, where the distinguished universal arrow witnessing that composition composes with $a$ in $V$ into a $t$-arrow from $a$ to $t(a)$.\footnote{The notation $t(a)$ means the target of $a$; not to be confused with the operation $t \in T(1)_c$. We do not believe this unfortunate overload of notation occurs anywhere else in this paper.} This $t$-arrow then factors through the universal arrow out of $s(a)$ via an $\eta(c)$-arrow which we regard as the composite of $a$
	\item identities are given by identity arrows in $V$
	\item composition is given by composition in $V$
\end{itemize}
\end{proof}

Representable $T$-multicategories are equivalent to a weaker analogue of $T$-structured categories. This requires redefining the latter as algebras for a monad analogous to $T$ on the 2-category of $\Cat$-valued presheaves on $\C$. $T$-structured categories are the strict algebras for this 2-monad, while pseudo-algebras provide a notion of weak $T$-structured categories (as discussed in \cite[Section 6.6]{leinster}) which are equivalent to representable $T$-multicategories. When $T$ is the free monoid monad, this recovers the equivalence of monoidal categories and representable multicategories of \cite[Section 9]{Hermida}.

\begin{ex}
A uniformly representable $\fdc$-multicategory is by \cref{representable_structured} a strict triple category, and a representable $\fdc$-multicategory is a triple category which is weak in the horizontal and vertical directions, though strict in the forward direction.
\end{ex}

\begin{defn}
For $\E \to \D$ a further restriction of cell shapes, a $T$-multifunctor between $\D$-representable $T$-multicategories is \emph{$\E$-strong} if it preserves universal arrows when restricted to $\E$, and a $T$-multifunctor between uniformly $\D$-representable $T$-multicategories is \emph{$\E$-strict} if it preserves the distinguished universal $t$-arrows $h_a$ when restricted to $\E$. We say such a $T$-multifunctor is simply \emph{strong} (resp. \emph{strict}) when it is $\D$-strong (resp. $\D$-strict).
\end{defn}

\begin{ex}
For $V,V'$ weak monoidal pseudo-double categories, general $\MC$-multifunctors $V \to V'$ correspond to lax monoidal lax double functors, while strong $\MC$-multifunctors correspond to strong monoidal pseudo-double functors. When $V,V'$ are strict (in both senses), strict $\MC$-multifunctors correspond to strict monoidal double functors. 

Meanwhile, 0-strong (resp. 0-strict) $\MC$-multifunctors look like lax monoidal lax double functors which are strong (resp. strict) monoidal on objects.
\end{ex}

\subsection{Triviality plus representability as extra structure}

The ultimate specialization of $\fc$-multicategories from \cite[Table 2.1]{LeinsterEnrichment} is the combination of ``vertically trivial'' and ``(uniformly) representable,'' which recovers (strict) monoidal categories and, ultimately, the classical definition of enriched categories. The combination of these two conditions generally produces lower-dimensional data (by $\D$-triviality) with extra algebraic structure (from representability) that often recovers bases of enrichment already present in the literature.

\begin{ex}
Completing the table in \cite[Table 2.1]{LeinsterEnrichment}, a (uniformly) representable 0-trivial $\fc$-multicategory is a (strict) monoidal category, and a (uniformly) representable 0-discrete $\fc$-multicategory is a bicategory (2-category).
\end{ex}

\begin{ex}
A 0-trivial (uniformly) representable $\MC$-multicategory is a category with two compatible (strict) monoidal structures, which by a standard Eckmann-Hilton type argument is the same as a braided (strict symmetric) monoidal category. A 0-discrete representable $\MC$-multicategory is analogously to a monoidal bicategory which is strict monoidal on objects, and a uniformly representable 0-discrete $\MC$-multicategory corresponds to a strict monoidal 2-category.
\end{ex}

\begin{ex}
When $u : \D \to \C$ is a restriction of cell shapes away from a choice $e$ of endpoint objects (\cref{endpoints}) and $T$ is a $\D$-graded familial monad on $\ch$, $\D$-trivial (uniformly) representable $T$-multicategories are equivalent to a (strict) $(T,e)$-structured category in the sense of \cite[Section 8.1]{mythesis}. Again using \cite[Corollary C.4.8]{leinster}, when $T$ is finitary over $e$ this includes (albeit non-canonically) any symmetric monoidal category.
\end{ex}

\begin{ex}
An example of a restriction of cell shapes away from an endpoint object is $\G_1 \vee \G_1$ inside $\G_1 \times \G_1$, and a $(\G_1 \vee \G_1)$-trivial representable $\fdc$-multicategory is a weak triple category with a unique object, vertical morphism, horizontal morphism, forward morphism, vertical-forward square, and horizontal-forward square. What remains are the vertical-horizontal squares and the cubes, which form the objects and morphisms of a category. This category has two weak monoidal structures given by horizontal and vertical composition which have isomorphic units and satisfy the interchange law, so by an Eckmann-Hilton type argument the two monoidal structures agree up to isomorphism and the resulting monoidal category is braided. Similarly a $(\G_1 \vee \G_1)$-trivial, uniformly representable $\fdc$-multicategory is a category with two strict mooidal structures satisfying interchange, resulting in a strict symmetric monoidal category.
\end{ex}

\begin{ex}
If we relax the definition of representability so that universal arrows need not be closed under composition, we see that a duoidal category $M$ (strict/weak) also provides a $(\G_1 \vee \G_1)$-trivial (uniformly/not uniformly) representable $\fdc$-multicategory: the squares are the objects of $M$, a composite of universal arrows goes from an $n \times m$ grid of objects in $M$ to any possible duoidal composite of that grid, and an $(n,m)$-arrow is then any composite of universal arrows followed by a morphism in $M$. Horizontal composition and identities are given by $(\star,J)$ and vertical composition and identities by $(\diamond,I)$. Every grid has a universal arrow to the vertical composite of its horizontal composites, and these are closed under vertical composition but not horizontal composition.
\end{ex}

\begin{ex}\label{verticallytrivialrep}
A $(\G_1 \times 0)$-trivial (or \emph{vertically trivial}) representable $\fdc$-multicategory $V$ is a weak triple category with just one object, one vertical morphism, one forward morphism, and one vertical-forward square. This is precisely the data of a double category whose objects are the horizontal morphisms of $V$, vertical morphisms are the vertical-horizontal squares, forward morphisms are the horizontal-forward squares in $V$, and squares are the cubes in $V$. Furthermore, horizontal composition in $V$ makes this a monoidal double category.
\end{ex}

\section{Enrichment}\label{sec:enrichment}

Throughout this section, $u : \D \to \C$ will denote a restriction of cell shapes, $T$ a $\D$-graded familial monad on $\ch$, and $V$ will denote a $T$-multicategory. We will also sometimes consider a further restriction of cell shapes $v : \E \to \D$

\begin{defn}\label{def:enrichment}
A $(V,\D)$-enriched $T$-algebra consists of a $T_\D$-algebra $A$ and a $T$-multifunctor $\Homm : Mu_\ast A \to V$. 
\end{defn}

Unpacking this, an enriched $T$-algebra with the cell shapes of $\D$ regarded as lower-dimensional consists of:
\begin{itemize}
	\item a $T_\D$-algebra $A$ 
	\item for each $d$-cell in $A$, a $d$-cell in $V$ with an arrow in $V$ from ay composable arrangement of these cells to their composite. These cells can be considered as ``book-keeping,'' merely recording the types of objects that will describe the higher dimensional cells in the enriched $T$-algebra
	\item for each $c$ in $\C$ but not in $\D$, and each possible $c$-cell position in $A$ (that is, each map $u^\ast y(c) \to A$ in $\dh$), a $c$-cell in $V$ whose boundary $d$-cells agree under $\Homm$ with the corresponding boundary cells in $A$. These are the ``Hom-objects'' of the enriched $T$-algebra, which are closest to the classical setting when the $c$-cells of $V$ are similar to sets with additional structure. The lower dimensional book-keeping for the $d$-cells in $A$ determines which $c$-cells are eligible to be a particular Hom-object
	\item for each $t \in T(1)_c$ and $a : u^\ast T[t] \to A$ which composes on its boundary into $b : u^\ast y(c) \to A$, a ``composition map'' $t$-arrow in $V$ from $\Homm(a) : Mu_\ast u^\ast T[t] \to Mu_\ast A \to V$ to $\Homm(b)$. This map from the many Hom-objects ($c$-cells) in $V$ making up $\Homm(a)$ to the single Hom-object $\Homm(b)$ is analogous to the map from $\Homm(x,y) \otimes \Homm(y,z)$ to $\Homm(x,z)$ in a classical enriched category.
\end{itemize}

\begin{defn}
A map of $(V,\D)$-enriched $T$-algebras from $(A,\Homm)$ to $(A',H')$ is a map of $T_\D$-algebras $A \to A'$ along with a transformation of $T$-multifunctors as below:
\bctikz
Mu_\ast A \ar{dr}[swap]{H} \ar{rr} \ar[Rightarrow, shorten=16, shift right=5]{rr} & & Mu_\ast A' \ar{dl}{H'} \\
& V
\ectikz
\end{defn}

\begin{ex}\label{enriched_cats}
A category ($\fc$-algebra) enriched in an $\fc$-multicategory $V$ is precisely as described in \cite[Section 2.2]{LeinsterEnrichment}. It consists of: 
\begin{itemize}
	\item a set ($\fc_0$-algebra) $A$ 
	\item vertices $\Homm(a)$ in $V$ for each element $a$ of $A$
	\item edges $\Homm(a,b)$ in $V$ from $\Homm(a)$ to $\Homm(b)$ for each pair of objects $a,b$ in $A$
	\item composition $n$-arrows from $\Homm(a_0,a_1),...,\Homm(a_{n-1},a_n)$ to $\Homm(a_0,a_n)$ for each $n \in \nats$ and $a_0,...,a_n \in A$, satisfying unit and associativity equations
\end{itemize}
\end{ex}

\begin{ex}
If $\D$ is the empty category $\varnothing$, $T_\varnothing$ is the identity monad on the terminal category $\widehat{\varnothing}$. For $A$ the unique $T_\varnothing$-algebra, $u_\ast A$ is the terminal $T$-algebra and $Mu_\ast A$ is the terminal $T$-multicategory. A $(V,\varnothing)$-enriched $T$-algebra then consists of:
\begin{itemize}
	\item for each $c$ in $\C$, a $c$-cell $\Homm(c)$ in $V$ forming a map $\Homm_0 : * \to V_0$ in $\ch$
	\item for each $t \in T(1)_c$, a single $t$-arrow from $\Homm_0 \circ ! : T[t] \to * \to V_0$ to $\Homm(c)$, natural in $c$
	\item such that these arrows are closed under identities and composites
\end{itemize}

For instance a $(V,\varnothing)$-enriched category is then a horizontal monoid in the virtual double category $V$, meaning a horizontal endomorphism $m$ equipped with squares from the $n$th iterated composite of $m$ to $m$ for $n \ge 0$ closed under composition in $V$.
\end{ex}

\begin{ex}
When $u : \D \to \C$ is a restriction of cell shapes away from a choice $e$ of endpoint objects (\cref{endpoints}), $T$ is a $\D$-graded familial monad on $\ch$, and $V$ is $\D$-trivial and representable (that is, a $(T,e)$-structured category as in \cite[Section 8.1]{mythesis}), a $(V,\D)$-enriched $T$-algebra is the same as a $V$-enriched $T$-algebra in the sense of \cite[Section 8.2]{mythesis}.
\end{ex}

\begin{ex}\label{enriched_moncats}
We call a $(V,\G_0)$-enriched $\MC$-algebra simply a $V$-enriched monoidal category. It consists of:
\begin{itemize}
	\item A monoid $A$
	\item For each $a \in A$ a vertex $\Homm(a)$ in $V_0$
	\item For each $a,a' \in A$, an edge $\Homm(a,a')$ from $\Homm(a)$ to $\Homm(a')$ in $V$. This is because an edge in $u_\ast A$ is determined by its source and target
	\item For each $a_1,...,a_n \in A$, an $n$-to-1 arrow from $\Homm(a_1),...,\Homm(a_n)$ to $\Homm(a_1 \cdots a_n)$ in $V$. This includes a 0-to-1 arrow to $\Homm(e)$ for $e$ the unit of $A$
	\item For each list of elements $a_{1,0},...,a_{1,m_1},...,a_{n,0},...,a_{n,m_n} \in A$, an arrow in $V$ from the arrangement of paths of the form 
$$\Homm(a_{i,0},a_{i,1}),...,\Homm(a_{i,m_i-1},a_{i,m_i})$$ to the edge $$\Homm(a_{1,0} \cdots a_{n,0},a_{1,m_1} \cdots a_{n,m_n}).$$
	\item These arrows are closed under identities and composition in $V$
\end{itemize}
This definition is rather complicated, but simplifies in the case when $V$ is representable, which is to say a monoidal double category (we will assume it is weak and pseudo- henceforth unless otherwise specified) where we can call the edges ``morphisms.''. In that case, a $V$-enriched monoidal category consists of:
\begin{itemize}
	\item A monoid $A$
	\item For each $a \in A$ an object $\Homm(a)$ in $V$
	\item For each $a,a' \in A$, a morphism $\Homm(a,a')$ in $V$ from $\Homm(a)$ to $\Homm(a')$
	\item A forward arrow $I \to \Homm(e)$ for $e$ the unit of $A$ and $I$ the unit of $V$
	\item For each $a_1,a_2 \in A$, a forward arrow $\Homm(a_1) \otimes \Homm(a_2) \to \Homm(a_1 a_2)$ in $V$
	\item These forward arrows satisfy the unit and associativity equations of a lax monoidal functor from the discrete monoidal category $A$ to the forward monoidal category of $V$
	\item For each $a_1,a_2,a'_1,a'_2 \in A$, a square
\bctikz[column sep=80]
\Homm(a_1) \otimes \Homm(a_2) \ar[""{name=S, below}]{r}{\Homm(a_1,a'_1) \otimes \Homm(a_2,a'_2)} \dar & \Homm(a'_1) \otimes \Homm(a'_2) \dar \\
\Homm(a_1a_2) \ar[""{name=T, above}]{r}[swap]{\Homm(a_1a_2,a'_1a'_2)} & \Homm(a'_1a'_2)
\arrow[Rightarrow,shorten=5,from=S,to=T]
\ectikz
	\item For each $a,a',a'' \in A$, a square
\bctikz
\Homm(a) \rar{\Homm(a,a')} \dar[equals] & \Homm(a') \rar{\Homm(a',a'')} & \Homm(a'') \dar[equals] \\
\Homm(a) \ar[""{name=T, above}]{rr}[swap]{\Homm(a,a'')} & & \Homm(a'')
\arrow[Rightarrow,shorten=5,from=1-2,to=T]
\ectikz
	\item These squares satisfy unit and associativity equations with respect to both products and compositions, as well as an interchange equation
\end{itemize}
This data can be summarized as a monoid $A$ and a lax monoidal lax double functor from the monoidal category\footnote{Here $u_\ast A$ is regarded as a double category with only identity forward morphisms and squares} $u_\ast A$ to $V$.
\end{ex}

The book-keeping in $V$ for the lower dimensional cells has thus far only assumed a $t$-arrow in $V$ from $\Homm(a)$ to $\Homm(b)$ for $t \in T(1)_d$, $a : T[t] \to A$, and $b \in A_d$, meaning a lax monoidal functor to $V$ in the previous example. But when $V$ is representable or uniformly representable, we can impose universality conditions on these arrows that correspond to the book-keeping part of $\Homm$ weakly or strictly preserving $T$-algebra structure (so replacing this lax monoidality with strong or strict monoidality).

\begin{defn}\label{strong_strict}
When $V$ is (uniformly) $\E$-representable, a $(V,\D)$-enriched $T$-algebra $(A,\Homm)$ is $\E$-strong (resp. $\E$-strict) if $\Homm$ is $\E$-strong (resp. $\E$-strict). We say $(A,\Homm)$ is simply \emph{strong} (resp. \emph{strict}) when it is $\D$-strong (resp. $\D$-strict).
\end{defn}

\begin{ex}
A $V$-enriched monoidal category $(A,\Homm)$, where $V$ is a monoidal double category, is 0-strong when the vertical morphisms $\Homm(a_1) \otimes \Homm(a_2) \to \Homm(a_1a_2)$ are isomorphisms. When $V$ is strict monoidal, $(A,\Homm)$ is 0-strict when these are in fact identities. 

This may seem unusual from the perspective of lax double functors (it is uncommon to consider functors which are strong monoidal on objects but lax monoidal on morphisms), but from an enrichment point of view it is fairly natural, as it is reasonable to expect the monoid $A$ of objects in an enriched monoidal category to maintain its form in $V$. In other words, the emphasis in enrichment is on modeling only the higher dimensional cells in $V$ using the $\Homm(a,a')$'s, so it is expected that there would be a non-invertible map 
$$\Homm(a_1,a'_1) \otimes \Homm(a_2,a'_2) \to \Homm(a_1a_2,a'_1a'_2),$$
describing how morphisms are tensored together. The objects $\Homm(a)$ do not carry the same interpretation as collections of cells, instead serving more of a book-keeping role. They merely allow the $\Homm(a,a')$'s to live in different categories when desireable, so there is no intuitive reason to expect that the maps $\Homm(a_1) \otimes \Homm(a_2) \to \Homm(a_1a_2)$ should not be isomorphisms or even identities. 
\end{ex}

\begin{ex}
When $M$ is a duoidal category regarded as a representable $(\G_1 \vee \G_1)$-trivial $\fdc$-multicategory, an $(M,\G_1 \vee \G_1)$-enriched double category agrees with the definition of ++-enriched double category in \cite{AguiarEnrichment}, which amounts to a pair $A$ of 1-categories with the same objects, for each square boundary $\alpha$ in $A$ an object $\Homm(\alpha)$ in $M$, and morphisms in $M$ corresponding to identities and composition:
\begin{itemize}
	\item when $\alpha$ has identities as its horizontal arrows, a morphism $J \to \Homm(\alpha)$ in $M$
	\item when $\alpha$ has identities as its vertical arrows, a morphism $I \to \Homm(\alpha)$ in $M$
	\item when $\alpha,\alpha'$ are horizontally adjacent with composite square boundary $\alpha''$, a morphism $\Homm(\alpha) \star \Homm(\alpha') \to \Homm(\alpha'')$
	\item when $\alpha,\alpha'$ are vertically adjacent with composite square boundary $\alpha''$, a morphism $\Homm(\alpha) \diamond \Homm(\alpha') \to \Homm(\alpha'')$
\end{itemize}
These morphisms must satisfy unit, associativity, and interchange laws analogous to those in a double category, including a unit interchange equation which says that when $\alpha$ has identities as both its horizontal and vertical arrows, $I \to J \to \Homm(\alpha)$ agrees with $I \to \Homm(\alpha)$.

A double category enriched in a representable $\fdc$-multicategory $V$, namely a weak triple category, consists of forward-lax functors from the horizontal/vertical categories of a shared-object pair of categories $A$ to the horizontal/vertical categories of $V$, along with horizontal-vertical squares for each square boundary in $A$ and composition cubes. It is strong when these are pseudo-functors, and when $V$ is uniformly representable (a strict triple category) the enriched double category is strict when these are ordinary functors. 
\end{ex}

\begin{ex}\label{vertenricheddoublecats}
For $V$ a virtual triple category, $u : \G_1 \times 0 \to \G_1 \times \G_1$, and $A$ a category, a virtual triple functor $Mu_\ast A \to V$ amounts to suitable choices of objects and vertical arrows for those in $A$ with forward-vertical squares witnessing each composition in $A$, a horizontal arrow for each pair of objects in $A$, a square for each pair of arrows in $A$, and squares/cubes in the forward direction of $V$ for each composition map in this \emph{vertically $V$-enriched} double category. 

When $V$ is representable and vertically trivial, hence a monoidal double category (\cref{verticallytrivialrep}), a vertically $V$-enriched double category amounts to 
\begin{itemize}
	\item a category $A$
	\item for each pair of objects $a,b$ in $A$, an object $\Homm(a,b)$ of $V$
	\item for each pair of morphisms $f : a \to b,f' : a' \to b'$ in $A$, a vertical arrow $\Homm(f,f')$ of $V$ from $\Homm(a,b)$ to $\Homm(a',b')$
	\item for each object $a$ in $A$, a forward arrow $I \to \Homm(a,a)$
	\item for each triple of objects $a,b,c$ in $A$, a forward arrow $\Homm(a,b) \otimes \Homm(b,c) \to \Homm(a,c)$
	\item for each morphism $f$ in $A$, a square from $\id_I$ to $\Homm(f,f)$
	\item for each triple of morphisms $f,f',f''$ in $A$, a square from $\Homm(f,f') \otimes \Homm(f',f'')$ to $\Homm(f,f'')$
	\item for each object $a$ in $A$, a forward-globular square from $\id_{\Homm(a,a)}$ to $\Homm(\id_a,\id_a)$ 
	\item for $a \atol{f} b \atol{g} c$ and $a' \atol{f'} b' \atol{g'} c'$ in $A$, a forward-globular square from $\Homm(g,g') \circ \Homm(f,f')$ to $\Homm(g \circ f,g' \circ f')$
	\item these composition maps satisfy unit, associativity, and interchange equations and commute with sources and targets in $A$
\end{itemize}
This definition is in fact equivalent to the following:
\begin{itemize}
	\item a category $A$
	\item a category enriched in the forward monoidal category of $V$ with the same objects as $A$
	\item a category enriched in the monoidal category of vertical arrows and squares in $V$ whose objects are the morphisms of $A$
	\item these enriched categories respect sources and targets, and laxly respect identities and composites in $A$
\end{itemize}

This definition is very nearly saying that ``a category internal to categories enriched in a category internal to monoidal categories is the same as a category internal to enriched categories.'' A statement like this could perhaps be made more formal given a suitable definition of morphism between enriched categories with varying enrichment bases, but we do not pursue this further.
\end{ex}






\appendix
\section{Enriched $T$-Multicategories}\label{sec:enrichedmulticats}

We describe how for $T$ a familial monad on $\ch$, $T$-multicategories are algebras for a familial monad $T^+$ on $\cplush$. The monad $T^+$ was defined by Leinster for any cartesian monad $T$ in \cite[Appendix A, Theorem 1.2.1]{LeinsterEnrichment}, and shown to be familial whenever $T$ is in \cite[Proposition 6.5.5]{leinster}. Here we give unwinded descriptions of the category $\C^+$ and the familial representation of $T^+$ for familial $T$. 

There is always a restriction of cell shapes $\C \to \C^+$, and we show that enriching $T$-multicategories with respect to this restriction recovers Leinster's definition of enrichment for $T$-multicategories.

\subsection{$\C^+$ and $T$-graphs}

\begin{defn}
For $T$ a familial monad on $\ch$, $\C^+$ is the category containing:
\begin{itemize}
	\item a copy of $\C$, whose objects represent ``$c$-cells'' 
	\item a disjoint copy of $\sint T(1)$, whose objects represent ``$t$-arrows.'' When $t \in T(1)_c$, for each face of the cell shape $c$ (given by a morphism $i : c' \to c$ in $\C$) there is a corresponding face $i_t : t_i \to t$, where $t_i = T(1)_i t \in T(1)_{c'}$
	\item for each $t \in T(1)_c$, a ``target'' face map $\tau_t : c \to t$ in $\C^+$ 
	\item for each $t$ in $T(1)$ and $x \in T[t]_d$, a ``source'' face map $\sigma_{x} : d \to t$ in $\C^+$
\end{itemize}
subject to the equations:
\begin{itemize}
	\item $\tau_t \circ i = i_t \circ \tau_{t_i} : c' \to t$ for each $i : c' \to c$ in $\C$ and $t \in T(1)_c$
	\item $\sigma_x \circ j = \sigma_{x_j} : d' \to t$ for each $t \in T(1)_c$, $j : d' \to d$ in $\C$, and $x \in T[t]_d$
	\item $i_t \circ \sigma_x = \sigma_{T[i_t]x} : d \to t$ for each $i : c' \to c$, $t \in T(1)_c$, and $x \in T[t_i]_d$
\end{itemize}
\end{defn}

There is a canonical restriction of cell shapes $u : \C \to \C^+$, as there are no morphisms in $\C^+$ from the objects in $\sint T(1)$ to those of $\C$. 

\begin{rem}
A more concise definition of $\C^+$ which differs from that of Leinster in \cite[Proposition 6.5.5]{leinster} is given by the ``Grothendieck Construction'' for functors $\G_1 \to \Prof$, the category of profunctors. Any functor $\two \to \Prof$ has a corresponding collage category over $\two$, and so a pair of parallel profunctors yields a category over $\G_1$. In this case, consider the ``polynomial'' or ``bridge diagram'' representation of $T$ first described in \cite[Remark 2.12]{WeberPra} and further discussed surrounding \cite[Definition B.4]{representability}. Each familial monad $T$ determines a functor $\sint T(1) \to \C$ and a span $\C \afm \soint T[-] \to \sint T(1)$, both of which induce profunctors from $\C$ to $\sint T(1)$. The Grothendieck Construction of this pair yields $\C^+$, and the resulting functor $\C^+ \to \G_1$ sends all of $\C$ to $\O$, all of $\sint T(1)$ to $\I$, the target maps to $t : \O \to \I$, and the source maps to $s : \O \to \I$.
\end{rem}

\begin{ex}
When $\C=\one$ and $T=\id_\Set$, $\C^+$ is $\G_1$, as the identity monad on sets is represented by a single operation with the singleton set as its arity, hence the unique source morphism from $\O$ to $\I$ in $\G_1$.
\end{ex}

\begin{ex}
When $\C=\one$ and $T=list$, the free monoid monad, $\C^+$ consists of a vertex $0$ as well as $n$-to-1 arrows for all $n \in \nats$, each with $n$ source maps from $0$ and one target map from $0$. Presheaves on this category are many-to-one graphs, which underly plain multicategories.
\end{ex}

More generally, presheaves in $\cplush$ admit a description as spans.

\begin{defn}\label{tgraph}
A $T$-graph is a span in $\ch$ of the form
\bctikz
& \ar{dl}[swap]{dom} Y \ar{dr}{cod} \\
TX & & X,
\ectikz
and a morphism of $T$-graphs is a map of spans such that the left leg is in the image of $T$.
\end{defn}

\begin{prop}\label{presheavestgraphs}
The category of $T$-graphs is equivalent to $\cplush$.
\end{prop}

\begin{proof}
(Sketch) Given a $T$-graph as in \cref{tgraph}, we define a presheaf on $\C^+$ as follows:
\begin{itemize}
	\item the $c$-cells and their faces are those of $X$
	\item the $t$-arrows for $t \in T(1)_c$ are the $c$-cells of $Y$ with domain of the form $T[t] \to X$ in $TX_c$
	\item the arrow faces of a $t$-arrow are given by the corresponding face maps in $Y$
	\item the target face of a $t$-arrow is given by its codomain $c$-cell in the $T$-graph
	\item the source faces of a $t$-arrow are given by the $d$-cells in its domain $T[t] \to X$
\end{itemize}
Conversely, given a presheaf $V$ in $\cplush$ we define a $T$-graph as follows:
\begin{itemize}
	\item let $X$ be the restriction $u^\ast V$ along the canonical restriction of cell shapes of $\C^+$. In other words, $X$ contains the $c$-cells of $V$
	\item let $Y$ be the left Kan extension of the restriction of $V$ to $\sint T(1)$ along the discrete fibration $\sint T(1) \to \C$. In other words, $Y_c$ is the disjoint union of $t$-arrows in $V$ over all $t \in T(1)_c$
	\item the codomain map is determined by the targets of the $t$-arrows, and the domain map is determined by their sources, which for a $t$-arrow is precisely the desired map $T[t] \to X$
\end{itemize}
It is straightforward to check that these constructions are quasi-inverse to each other.
\end{proof}

\subsection{Free $T$-multicategories and enrichment}

We now define the ``free $T$-multicategory'' familial monad $T^+$, by first describing some notation for the composable operations in a $T$-multicategory.

\begin{defn}\label{sequence}
A \emph{length $n$ $(T,c)$-sequence} for $n \in \nats$ and $c$ in $\C$ is a sequence $(t_1,...,t_n) \in T(1)_c$ equipped with maps $f_k : T[t_k] \to T(1)$ such that $t_{k+1} = \mu(t_k,f_k)$ for $k=1,...,n-1$. We say $t_n$ is the \emph{outer operation} of the sequence, which we sometimes refer to as a decomposition of $t_n$. When $n=0$, we define its outer operation to be $\eta(c) \in T(1)_c$.
\end{defn}

These sequences will form the operations in the free $T$-multicategory monad. We now proceed to define their arities inductively.

\begin{defn}\label{sequencearities}
The unique length 0 $(T,c)$-sequence has arity $y(c)$ in $\cplush$. The length 1 $(T,c)$-sequence $(t)$ has arity $y(t)$ in $\cplush$, the representable $t$-arrow. We now proceed inductively on the assumption that the arity of a length $n$ $(T,c)$-sequence contains a ``source subdiagram'' of the form $u_!T[t_n]$, which is to say, a copy of $T[t_n]$ built out of $c$-cells in $\cplush$. Of course this assumption holds in the base case as $y(t)$ has a copy of $T[t]$ as its source cells.

Assuming this hypothesis for $n-1$ and given a length $n$ $(T,c)$-sequence $(t_1,...,t_n)$, we can construct the arity $T^+[t_1,...,t_{n-1}]$ of its length $n-1$ subsequence $(t_1,...,t_{n-1})$. We now define the arity of $(t_1,...,t_n)$ as the pushout under $u_!T[t_{n-1}]$ of $T^+[t_1,...,t_{n-1}]$ along its source subdiagram and 
$$\colim \left( \sint T[t_{n-1}] \atol{\sint f_{n-1}} \sint T(1) \emb C^+ \emb \cplush \right)$$
along its target subdiagram 
$$\colim_{x \in T[t_{n-1}]_c} y(f_{k-1}(x)) \afm \colim_{x \in T[t_{n-1}]_c} y(c) \cong u_!T[t_{n-1}].$$
\end{defn}

\begin{ex}
An $(\fc,1)$-sequence of length 2 has as its arity the diagram in depicted below. Longer sequences would have aritities given by similarly shaped towers of $n$-to-1 squares with greater height. These are the composable diagrams in an $\fc$-multicategory.
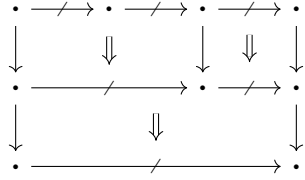
\[
\begin{tikzcd}
\mdot \dar \ar[slash]{r} & \mdot \ar[slash]{r} & \mdot \dar \ar[slash, ""{name=D, below}]{r} & \mdot \dar \\
\mdot \dar \ar[slash, ""{name=A, above}]{rr} & {\color{white}{\mdot}} \ar[phantom, ""{name=B, below}]{r} & \mdot \ar[slash, ""{name=E, above}]{r} & \mdot \dar \\
\mdot \ar[slash, ""{name=C, above}]{rrr} & & & \mdot 
\arrow[Rightarrow,shorten=4,from=1-2,to=A]
\arrow[Rightarrow,shorten=5,from=B,to=C]
\arrow[Rightarrow,shorten=5,from=D,to=E]
\end{tikzcd}
\]
This diagram can be defined as the pushout of the lower square and the upper row of squares along the pair of adjacent edges connecting them, where the upper row of squares is a colimit of the form in \cref{sequencearities}. The outer operation of this sequence is $3 \in \fc(1)_1$, and as such the diagram resembles a decomposition of a 3-to-1 square.
\end{ex}

\begin{defn}
The monad $T^+$ on $\cplush$ is represented as follows:
\begin{itemize}
	\item $T^+(1)_c = \{*_c\}$ for all $c$ in $\C$, with $T^+[*_c] = y(c)$ the representable $c$-cell in $\cplush$
	\item $T^+(1)_t$ for $t \in T(1)_c$ is the set of $(T,c)$-sequences of any length with outer operation $t$.  The corresponding arities $T^+[t_1,...,t_n=t]$ are as defined in \cref{sequencearities}.
	\item the unit operations for $c$-cells are the unique such operaitons, and for $t$-arrows are given by the sequences $(t)$ of length 1 and representable arity
	\item multiplication is defined by observing that a map $T^+[t_1,...,t_n] \to T^+(1)$ in $\cplush$ amounts to choices of decompositions of each $t'$-arrow in $T^+[t_1,...,t_n]$ such that all $t'$-arrows in the $k$th row are assigned decompositions of length $m_k$. Transposing this assignment results in $m_k$ maps $T[t_k] \to T(1)$ in $\ch$ which after applying the multiplication $\mu$ for $T$ form a $(T,c)$-sequence of length $m_k$. These sequences for $k=1,...,n$ concatenate into a sequence of length $m_1 + \cdots + m_n$ in $T^+(1)_{t_n}$, which defines the multiplication and can be checked to have the the appropriate arity (see discussion following \cref{def:familial})
\end{itemize}
\end{defn}

It is straightforward to check that these operations are generated under the unit and multiplication constructions by the sequences of length 0 (which are only in $T^+(1)_{\eta(c)}$) and length 2, which are precisely the identities and compositions listed in the expanded definition of $T$-multicategories following \cref{def:multicat}. Hence algebras for this monad are $T$-multicategories, which are defined as $T$-graphs with the identities and composites indexed by the length 0 and 2 sequences satisfying the same unit and associativity equations determined by the definition of multiplication for $T^+$.

\begin{ex}
When $T=list$, the free monoid monad on sets, $T^+$ is the free multicategory monad $\fm$ on many-to-one graphs. Indeed, the trees which make up its operations and arities (see \cite[Example 2.14]{WeberPra}, ignoring the symmetries) are exactly the $list$-sequences from \cref{sequence}, where the height of a tree is the length of the corresponding sequence, and each term in the sequence is the number of vertices in each level of the tree (a number regarded as an element of $list(1) \cong \nats$).
\end{ex}

The definition of $T^+$ makes clear that it is $\C$-graded, as the operations for $c$-cells have representable arities arising from $\ch$. The restricted monad $(T^+)_\C$ is simply the identity monad on $\ch$, having a single operation for each $c$ in $\C$ with representable arity $y(c)$.

\begin{lem}
For $u : \C \to \C^+$ the canonical restriction of cell shapes and $A$ in $\ch$, $u_\ast A$ is isomorphic to the $T$-multicategory given by the span
$$TA \afml{\pi_1} TA \times A \atol{\pi_2} A.$$
\end{lem}

\begin{proof}
This follows from the observation that for each $t \in T(1)_c$, $u^\ast y(t)$ in $\ch$ is isomorphic to the disjoint union of the arity $T[t]$ and the representable $y(c)$. Therefore, $(u_\ast A)_t = \Hom(T[t],A) \times A_c$, so as products distribute over disjoint unions and $TA_c = \coprodl_{t \in T(1)_c} \Hom(T[t],A)$, applying the construction of a $T$-graph from $u_\ast A$ as in \cref{presheavestgraphs} yields $TA \times A$ as the apex of the $T$-graph, with the domain map given by projection to $TA$ and the codomain by projection to $A$ as desired. 
\end{proof}

We can now conclude that for $V$ a $T^+$-multicategory, a $(V,\C)$-enriched $T^+$-algebra in the sense of \cref{def:enrichment} is the same as a $V$-enriched $T$-multicategory in the sense of \cite[Definition 1.3.1]{LeinsterEnrichment}. That is, $Mu_\ast A$, for $u$ the canonical restriction of cell shapes $\C \to \C^+$, is the same as $M$ applied to the indiscrete $T$-multicategory defined in \cite{LeinsterEnrichment}, so both notions of enriched $T$-multicategory are defined as a $T^+$-multifunctor from $Mu_\ast A$ to $V$.

\begin{ex}\label{enrichedplainmulticats}
When $\C=\one$ and $T=list$, our definition of enrichment then says that $(V,\one)$-enriched multicategories are the same as Leinster's $V$-enriched multicategories, for $V$ an $\fm$-multicategory (\cite[Section 3]{LeinsterEnrichment}). They consist of a set $A$ of objects, an object $\Homm(a)$ in $V$ for each $a \in A$, an $n$-to-1 edge in $V$ from $\Homm(a_1),...,\Homm(a_n)$ to $\Homm(b)$ for each $a_1,...,a_n,b \in A$, and identities and composite arrows in $V$ corresponding to the structure of a plain multicategory. 

Multicategories can, like many other structures we have discussed, be enriched in a monoidal double category, by noting that representable $\fm$-multicategories are double multicategories, namely category objects in the category of multicategories. Restricting those multicategories to representable multicategories, which are in turn monoidal categories, recovers monoidal double categories as a special case of $\fm$-multicategories which are in this sense ``doubly representable.'' Examples of multicategories enriched in a monoidal double category can be found in \cite{adaptives}.
\end{ex}

\begin{ex}\label{weakenrichment}
While we defer a detailed description of this to future work, a ``$n$-category-enriched $T$-multicategory'' is a convenient base of enrichment for when one wants the equations between composition operations in an enriched $T$-algebra to hold only up to coherent (higher) isomorphism. When, for instance, there is a category (or $n$-category) of fixed-boundary $t$-arrows in a $T$-multicategory rather than a set, it is straightforward to demand that two sides of an equation between such $t$-arrows are merely isomorphic (or equivalent) rather than equal.

This can be defined for any version of strict, weak, or semi-strict $n$-categories, though what precisely the $T^+$-multicategory of $n$-categories should look like is rather subtle in full generality. In many cases however, the symmetric monoidal category of $n$-categories is sufficient to define such a $T^+$-multicategory.
\end{ex}




%

\section{Structures Enrichable in a Monoidal Double Category}\label{app:mdcat}

Here we list various monads $T$ for which $T$-algebras are enrichable in a monoidal double category (presumed to be weak in both the monoidal and horizontal directions) and recall for each how monoidal double categories relate to $T$-multicategories.

\begin{itemize}

	\item Categories can be enriched in $\fc$-multicategories (\cref{enriched_cats}), which include monoidal double categories by forgetting the monoidal structure, as double categories are representable $\fc$-multicategories.

	\item As discussed in \cref{doublemulticats} and \cref{enrichedplainmulticats}, plain multicategories can be enriched in $\fm$-multicategories, and monoidal double categories are representable representable $\fm$-multicategories, or equivalently representable double multicategories, meaning categories internal to representable multicategories.

	\item Monoidal categories can be enriched in $\MC$-multicategories, and monoidal double categories are representable $\MC$-multicategories as discussed in \cref{enriched_moncats}.

	\item Double categories can be horizontally enriched in $\fdc$-multicategories, and monoidal double categories are representable, vertically trivial $\fdc$-multicategories as discussed in \cref{vertenricheddoublecats}. 

\end{itemize}

\bibliography{mybib}
\bibliographystyle{alpha}

\end{document}